\documentclass[11pt]{amsart}
\usepackage{amssymb}
\usepackage{amscd}
\usepackage{verbatim}
\usepackage{amssymb,amsfonts,pstricks,epsf}
\usepackage{graphicx}
\usepackage{epsfig}
\usepackage{color}
\usepackage{stmaryrd}
\usepackage{mathrsfs}
\usepackage{marginnote}

\newtheorem{theorem}{Theorem}[section]

\newtheorem{corollary}[theorem]{Corollary}

\definecolor{plum}{rgb}{1.0, 0.0, 1.0}

\begin{document}

\title[Comparison Results, Exit Time Moments, and Eigenvalues]{Comparison Results, Exit Time Moments, and Eigenvalues on Riemannian Manifolds with a lower Ricci Curvature Bound}

\author{Don Colladay, Jeffrey J. Langford, and Patrick McDonald}

\address{Department of Mathematics, Bucknell University, Lewisburg, Pennsylvania 17837}

\email{jeffrey.langford@bucknell.edu}

\address{Division of Natural Science, New College of Florida, Sarasota, FL 34243}

\email{colladay@ncf.edu}
\email{mcdonald@ncf.edu}

\date{\today}

\begin{abstract}
We study the relationship between the geometry of smoothly bounded domains in complete Riemannian manifolds and the associated sequence of $L^1$-norms of exit time moments for Brownian motion.  We establish bounds for Dirichlet eigenvalues and, for closed manifolds, we establish a comparison result for elements of the moment sequence using lower bounds on Ricci curvature. 
\end{abstract}



\maketitle

\section{Introduction}
The relationship between the geometry of a complete Riemannian manifold and properties of the associated collection of Brownian paths is the subject of many papers.  Driven in part by the attraction of developing new tools and intuitions for old problems, the associated literature has  grown steadily with a variety of cross-fertilizations appearing over the course of the last twenty years.  This paper contributes to efforts in this direction: we study the comparison geometry and spectral geometry of smoothly bounded precompact domains in complete Riemannian manifolds using a collection of invariants that arise naturally in the context of probability.  These invariants are constructed using Brownian paths and the volume form; they are naturally connected to the heat kernel and the Dirichlet spectrum.  Our results shed light on this connection.

 To formulate our results we need some notation.  Let $(M, g)$ be a complete Riemannian manifold and let $\Omega\subseteq M$ be a smoothly bounded precompact domain.  We will denote by $X_t$ Brownian motion on $M$ and by ${\mathbb P}^x$ the associated probability measure charging Brownian paths beginning at $x\in M.$  Let $\tau $ be the first exit time of $X_t$ from $\Omega$:
\[
\tau = \inf\{t\geq 0 : X_t \notin \Omega\}.
\]
Then $\tau$ is a random variable whose expectation with respect to the measure ${\mathbb P}^x$ solves a Poisson problem on $\Omega.$  More precisely, writing $u_1(x) = {\mathbb E}^x[\tau],$ we have that $u_1$ satisfies 
\begin{equation}\label{poission1.1}
-\Delta_g u_1=1  \quad \textup{in } \Omega,\qquad u_1=0\quad \textup{on }\partial \Omega.
\end{equation}
  
If $dV_g$ denotes the volume form on $M$ and we integrate over $\Omega$, we obtain an invariant of the domain:
\[
T_1(\Omega) = \int_{\Omega} u_1(x)\, dV_g(x).
\]
The invariant $T_1(\Omega)$ is called the {\it torsional rigidity} of the domain $\Omega$ and it has a long history.   First studied in the nineteenth century as part of the theory of elastic bodies, it exhibits properties analogous to those of the first Dirichlet eigenvalue (see \cite{P} for background on torsional rigidity and fundamental frequency).  Our first collection of invariants is a straightforward generalization of torsional rigidity obtained by integrating higher moments of the exit time.  Given a positive integer $n$, we define 
\begin{equation}\label{kthmoment}
T_n(\Omega) =  \int_{\Omega} {\mathbb E}^x[\tau^n] \,dV_g(x).
\end{equation}
We call the collection $\{T_n(\Omega)\}_{n\in {\mathbb N}}$ the {\it $L^1$-moment spectrum of $\Omega.$}  Our results involve the degree to which the geometry of a bounded domain can be studied using this family of invariants.  

We can express $T_n(\Omega)$ using the heat kernel by integrating twice over the domain $\Omega$ (see Section 2).  Because the heat kernel can be written in terms of eigenfunctions for the Dirichlet problem, it is natural to consider a second collection of invariants indexed by the values of the Dirichlet spectrum associated to $\Omega.$  To proceed we need more notation.

Let $\textup{spec}(\Omega)$ denote the Dirichlet spectrum listed in increasing order, with multiplicity.  Given an eigenvalue $\lambda\in \textup{spec}(\Omega),$ let $E_\lambda$ be the eigenspace associated to $\lambda$, and let $a_\lambda^2$ be the square of the $L^2$-norm of the orthogonal projection of the constant function $1$ on $E_\lambda.$  Let $\textup{spec}^*(\Omega)$ be the collection of real numbers $\nu$ for which $\nu \in \textup{spec}(\Omega)$ and $a_\nu^2>0.$  Then, as explained in Section 2 below, the invariants $a_\nu^2$ satisfy
\[
\textup{Vol}_g(\Omega) = \sum_{\nu \in \textup{spec}^*(\Omega)} a_\nu^2.
\]
We can now state our first result:

\begin{theorem}\label{Thm:LnEst} Let $M$ be a complete Riemannian manifold and $\Omega\subseteq M$ a smoothly bounded precompact domain.  Let $\lambda_n$ be the $n$th Dirichlet eigenvalue and denote by $\textup{spec}^*(\Omega)$ the values of the Dirichlet spectrum for which the associated eigenspace is not orthogonal to constants.  Then, with $a_\nu^2$ as above, we have the estimate
\begin{comment}
\begin{align*}
\displaymode \lambda_n(\Omega) &\leq \frac{\frac{T_{2k-1}(\Omega)}{(2k-1)!} - {\frac{T_{2k}(\Omega)}{(2k)!} - \sum_{j=1}^{n-1}\frac{a_{\lambda_j}^2}{\lambda_j^{2k}}}
\end{align*}
\end{comment}
\begin{equation}\label{Ineq:LambdanEst}
\lambda_n(\Omega) \leq 
\frac{ \displaystyle \frac{T_{2k-1}(\Omega)}{(2k-1)!} - \sum_{\underset{\nu <\lambda_n(\Omega)}{\nu \in \textup{spec}^{\ast}(\Omega)}} a_{\nu}^2\left(\frac{1}{\nu}\right)^{2k-1}}{\displaystyle \frac{T_{2k}(\Omega)}{(2k)!} -  \sum_{\underset{\nu <\lambda_n(\Omega)}{\nu \in \textup{spec}^{\ast}(\Omega)}} a_{\nu}^2\left(\frac{1}{\nu}\right)^{2k}}.
\end{equation}
Moreover, if $\lambda_n(\Omega)\in \textup{spec}^{\ast}(\Omega)$, the inequality becomes an equality in the limit as $k\to \infty$.
\end{theorem}

This result is an extension of the work of Dryden et al.  \cite{DLM} where the case $n$=1 was established without an equality claim.  The original motivation for the $n$=1 result involved applications in shape optimization and a sharpening of an inequality of P\'olya for Euclidean domains (see \cite{BBV} and \cite{BFNT}).  Similar results were obtained by Hurtado et al.  \cite{HMP3} in the context of warped product spaces; their work suggests that there should be a rich comparison geometry theory (see also \cite{HMP1}, \cite{HMP2}, \cite{Mc2}).  In the remainder of our paper, we develop this line of thought for spaces with Ricci curvature bounded below.  

A great deal is known about the structure of Riemannian manifolds with a lower Ricci curvature bound (see, for example, the survey \cite{W}).  For our purposes, lower bounds on Ricci curvature provide model spaces for comparison that, in turn, provide tools to establish estimates for our invariants.  Chief among the tools we employ is an isoperimetric result due to B\'erard et al. \cite{BBG}.  For closed (compact without boundary) Riemannian manifolds, the result of \cite{BBG} provides a Euclidean sphere, ${\mathbb S}^d(R)$, as a comparison space where the radius $R$ depends only on the dimension, the diameter, and the Ricci bound (see Theorem \ref{Thm:IsoIneq}).   We prove:

 \begin{theorem}\label{Thm:MSI}
Let $(M,g)$ be a connected $d$-dimensional closed Riemannian manifold with Ricci curvature bounded below by $(d-1)K,$ $K\in {\mathbb R},$ and let  $({\mathbb S}^d(R),g_0)$ be the Euclidean sphere prescribed by \cite{BBG}.  Let $\Omega\subseteq M$ be a smoothly bounded domain and let $\Omega^*$ be a geodesic ball in ${\mathbb S}^d(R)$ satisfying $\frac{\textup{Vol}_g(\Omega)}{\textup{Vol}_g(M)} = \frac{\textup{Vol}_{g_0}(\Omega^*)}{\textup{Vol}_{g_0}(\mathbb{S}^d(R))}.$  Then the moment spectra satisfy the inequality
 \begin{equation}\label{Ineq:MomSp}
 \frac{T_n(\Omega)}{\textup{Vol}_g(\Omega)}\leq\frac{T_n(\Omega^{\ast})}{\textup{Vol}_{g_0}(\Omega^{\ast})}
 \end{equation}
 for each $n\geq 1$.
 Moreover, if $K>0$ and $\Omega^{\ast \ast}$ is a geodesic ball in ${\mathbb S}^d\left(\frac{1}{\sqrt{K}}\right)$ satisfying $\frac{\textup{Vol}_g(\Omega)}{\textup{Vol}_g(M)} = \frac{\textup{Vol}_{g_1}(\Omega^{**})}{\textup{Vol}_{g_1}\left(\mathbb{S}^d\left(\frac{1}{\sqrt{K}}\right)\right)}$, then we also have
  \begin{equation}\label{Ineq:MomSp2}
 \frac{T_n(\Omega)}{\textup{Vol}_g(\Omega)}\leq\frac{T_n(\Omega^{\ast \ast})}{\textup{Vol}_{g_1}(\Omega^{\ast \ast})}.
 \end{equation}
 If equality holds in \eqref{Ineq:MomSp2} for some index $n$, then $M$ is isometric to the sphere $\mathbb{S}^d\left(\frac{1}{\sqrt{K}}\right)$ and $\Omega$ is isometric to a geodesic ball in $\mathbb{S}^d\left(\frac{1}{\sqrt{K}}\right)$.
  \end{theorem}

To place our work in the literature, we focus our remarks on material involving exit time and comparison geometry that shaped the development of our results.  In an early result in this direction, Debiard et al. \cite{DGM} studied the behavior of heat kernels on geodesic balls. The authors proved a theorem similar to Theorem \ref{Thm:MSI} in which they compared mean exit time for geodesic balls with mean exit time for geodesic balls in a space form.  For Euclidean domains, Aizenman and Simon \cite{AS} used the rearrangement result of Brascamp, Lieb, and Luttinger to prove that for given volume, pointwise mean exit time moments are bounded by the corresponding moments for Brownian motion starting at the center of a ball of the same volume.  In \cite{KMM}, the authors studied Euclidean domains using isoperimetric comparison, rearrangement results for elliptic PDE,  and a description of exit time moments as a solution of a hierarchy of Poisson problems to recover the above result of Aizenman and Simon and establish corresponding results for the $L^1$-moment spectrum.  Using isoperimetric comparison, the results of \cite{KMM} were extended to space forms in \cite{Mc1}.  In \cite{BS}, Burchard and Schmuckenschl\"ager studied the behavior of heat kernels for constant curvature space forms under symmetric rearrangement.  They used their results to bound exit time moments as above, they established the case of equality for domains in space forms, and they conjectured a result that implies Theorem \ref{Thm:MSI} (see Conjecture 4.11 of \cite{BS}).  Recently, Cadeddu et al. \cite{CGL} studied the optimization problem for the first exit time moment (torsional rigidity) and, using symmetric rearrangement, established comparison results under a variety of constraints involving bounded geometry, including the case of smoothly bounded precompact domains in manifolds with Ricci curvature bounded below.  Amongst the tools used in \cite{CGL} is the isoperimetric comparison result of \cite{BBG} cited above (see Section 2).  To establish our results, we use the isoperimetric comparison result of B\'erard et al. \cite{BBG}, symmetrization techniques in the spirit of Talenti \cite{T}, and the description of the $L^1$-moment spectrum in terms of a hierarchy of Poisson problems.    

As a corollary of Theorem \ref{Thm:LnEst} and the techniques used to establish the result, we establish a relationship between higher moments and the Cheeger constant (Theorem \ref{Thm:CSE}).  As a corollary of Theorem \ref{Thm:MSI}, we establish a Faber-Krahn theorem that illustrates how one might extract information contained in the higher moments (Corollary \ref{Cor:FK};  see also \cite{Mc2}).        

The remainder of this note is structured as follows.  In the next section we provide the required background involving exit time moments and symmetrization, including a discussion of the relationship between the $L^1$-moment spectrum and heat content.  In Section 3 we provide proofs of Theorems \ref{Thm:LnEst} and \ref{Thm:MSI} and the corollaries described above.  Along the way, we also establish a rearrangement result for elliptic boundary value problems that we believe is of independent interest (see Theorem \ref{Thm:CR}).

\section{Background}
\subsection{Exit time moments}\label{SubSect:ETM}
As in the Introduction, let $(M,g)$ denote a complete $d$-dimensional Riemannian manifold with $\Omega$ a smoothly bounded domain with compact closure.  Let $X_t$ denote Brownian motion in $M$ with infinitesimal generator $\Delta$ and for $x\in M$ let ${\mathbb P}^x$ denote the probability measure charging Brownian paths beginning at $x.$  Let $\tau$ denote the first exit time from $\Omega$:
\[
\tau = \inf  \{t\geq 0: X_t \notin \Omega\}.
\]
For $n$ a natural number, let $T_n(\Omega)$ be defined as in (\ref{kthmoment}):
\[
T_n(\Omega) =  \int_{\Omega} {\mathbb E}^x[\tau^n] \,dV_g(x),
\]
where ${\mathbb E}^x$ denotes expectation with respect to ${\mathbb P}^x$ and $dV_g(x)$ denotes the volume form.  The invariants $T_n(\Omega)$ are closely related to the heat content of $\Omega,$ a function constructed from the solution of an initial value problem on the domain $\Omega.$  More precisely, the solution of the initial value problem  
\begin{eqnarray}
u_t & = &  \Delta u \hbox{ in } (0,\infty) \times \Omega, \label{ivp1} \\
\lim_{t\to 0} u(t,x) & = & 1 \hbox{ in } \Omega, \label{ivp2} \\
\lim_{x\to \sigma} u(t,x) & = & 0 \hbox{ for } t\in (0,\infty) \hbox{ and } \sigma \in \partial \Omega, \label{ivp3}
\end{eqnarray}
can be written as
\begin{equation}\label{heatsolution}
u(t,x) = {\mathbb P}^x (\tau >t).
\end{equation}

The \emph{heat content} of $\Omega$ is the function $H:(0,\infty) \to {\mathbb R}$ defined by 
\begin{equation*}\label{heatcontent}
H(t) = \int_\Omega u(t,x) \,dV_g(x).
\end{equation*}
Using (\ref{heatsolution}) we can express moments of the exit time in terms of $u(t,x)$: 
\begin{equation}\label{etimeexpression}
{\mathbb E}^x[\tau^n] = n \int_0^\infty t^{n-1} u(t,x) \,dt.
\end{equation}
Combining (\ref{etimeexpression}) with Fubini's Theorem, we see that we can express the invariant $T_n(\Omega)$ as a moment of the heat content:
\begin{equation*}\label{emoments}
T_n(\Omega) = n\int_0^\infty t^{n-1} H(t) \,dt.
\end{equation*}

To elucidate the relationship between the $L^1$-moment spectrum and the Dirichlet spectrum, we write the solution of the initial value problem (\ref{ivp1})-(\ref{ivp3}) in terms of the Dirichlet kernel.   Let $\textup{spec}(\Omega)$ denote the Dirichlet spectrum of $\Omega$ listed in increasing order with multiplicity and fix a corresponding orthonormal basis of eigenfunctions, $\{\phi_\lambda: \lambda \in \textup{spec}(\Omega)\}.$  Then the Dirichlet heat kernel for $\Omega$ is given by 
\begin{equation*}\label{heatkernel}
p(t,x,y) = \sum_{\lambda \in \textup{spec}(\Omega)}  \phi_\lambda(x)\phi_\lambda(y)e^{-\lambda t} .
\end{equation*}
The heat content of $\Omega$ is then given by 
\begin{equation}\label{heatcontent2}
H(t) = \sum_{\lambda \in \textup{spec}(\Omega)} \left(\int_\Omega \phi_\lambda(x) \,dV_g(x) \right)^2 e^{-\lambda t}.
\end{equation}
We can rewrite the sum occurring in (\ref{heatcontent2}) as follows:   given a Dirichlet eigenvalue $\lambda$ with corresponding eigenspace $E_\lambda,$ write 
\begin{equation}\label{vp2}
a_\lambda^2 =  \sum_{\stackrel{\hat{\lambda} \in \textup{spec}(\Omega)}{\hat{\lambda}=\lambda} } \left(\int_\Omega \phi_{\hat{\lambda}}(x) \,dV_g(x) \right)^2.
\end{equation}
Then $a_\lambda^2$ is the square of the $L^2$-norm of the orthogonal projection of the constant function $1$ on the eigenspace $E_\lambda.$  We define a set of real numbers, $\textup{spec}^*(\Omega),$ by 
\begin{equation}\label{spec*}
\textup{spec}^*(\Omega) = \{ \lambda \in \textup{spec}(\Omega): \ a_\lambda^2 >0\}.
\end{equation}
Using (\ref{heatcontent2})-(\ref{spec*}), we can rewrite the heat content as 
\begin{equation}\label{Eqn:HC}
H(t) =  \sum_{\nu \in \textup{spec}^*(\Omega)} a_\nu^2 e^{-\nu t}.
\end{equation}

As mentioned in the Introduction, the sequence $\{a_\nu^2\}_{\nu \in \textup{spec}^*(\Omega)}$ is closely related to the volume of the domain $\Omega.$  To see this is the case, note that there is a small time asymptotic expansion of $H(t)$: 
\[
H(t) \simeq \sum_{k=0}^\infty h_k t^{\frac{k}{2}}, 
\]
where the coefficients are local geometric invariants (see \cite{BG}).  In particular, it is known that $h_0 = \textup{Vol}_g(\Omega)$ and we conclude 
\begin{equation}\label{Eqn:vp2}
\textup{Vol}_g(\Omega) =   \sum_{\nu \in \textup{spec}^*(\Omega)} a_\nu^2,
\end{equation}
from which we see that the $a_\nu^2$ {\it partition the volume of $\Omega.$}  

The most direct method for connecting the moment spectrum to the Dirichlet spectrum involves the study of the Mellin transform of the heat content.  The Mellin transform of $H(t)$ takes the form of a Dirichlet series
\begin{equation}\label{Eqn:Zeta}
\zeta(s) =    \sum_{\nu \in \textup{spec}^*(\Omega)} a_\nu^2 \left(\frac{1}{\nu}\right)^s
\end{equation}
and extends meromorphically to the plane with poles at the negative half-integers (see \cite{MM}).  The connection between the $L^1$-moment spectrum, the heat content, and the Dirichlet spectrum is embedded in the identity 
\begin{equation}\label{Eqn:ZetaTn}
\Gamma(n+1) \zeta(n) = T_n(\Omega).
\end{equation}

To extract information from (\ref{Eqn:ZetaTn}) we use recursion and a convenient relationship between $T_n(\Omega)$ and a hierarchy of Poisson problems.  More precisely, if we write 
\[
u_n(x) = {\mathbb E}^x[\tau^n],
\]
then we can apply the Laplace operator to the right hand side of (\ref{etimeexpression}) and integrate by parts to see that $u_n$ satisfies 
\begin{equation}
-\Delta_g u_n=nu_{n-1} \quad \textup{in }\Omega,\qquad u_n=0\quad \textup{on }\partial \Omega.\label{Eqn:PoisHier}
\end{equation}
This hierarchy is very useful in establishing Theorem \ref{Thm:LnEst}.  

\subsection{Ricci bounds and symmetrization}\label{SubSect:Symm}

To establish Theorem \ref{Thm:MSI}, we require results involving symmetrization and isoperimetric inequalities.  Throughout this section we assume $(M,g)$ is a connected and closed Riemannian manifold (so in particular $(M,g)$ is complete by Hopf-Rinow). Denote the Ricci curvature of $M$ by $\textup{Ric}_M,$ and let  
\[
R_{min}=\inf \{\textup{Ric}_{M}(u,u):u\in T_pM, \langle u,u \rangle_p=1,p\in M\}.
\]
B\'erard, Besson, and Gallot showed in \cite{BBG} (see also \cite{B} and \cite{CGL}) that closed Riemannian manifolds with a lower Ricci curvature bound admit an isoperimetric inequality:

%\marginnote{All three papers cited have different assumption on $M$.}[-3 cm]

%\marginnote{Do we need to mention that the isoperimetric profile on spheres is achieved by caps?}[-1 cm]

%\marginnote{It seems to me a trivial consequence that $\Omega$ is isometric to a ball. Is it?}[2 cm]

\begin{theorem}\label{Thm:IsoIneq} With $M$ as above, suppose the Ricci curvature on $M$ satisfies $R_{min}\geq (d-1)K$. Then there exists a $d$-dimensional sphere of radius $R$, denoted $\mathbb{S}^d(R)$, where for any smoothly bounded domain $\Omega$ in $M$, if $\Omega^{\ast}$ is a geodesic ball in $\mathbb{S}^d(R)$ satisfying
\begin{equation}\label{Eqn:VolNorm}
\frac{\textup{Vol}_g(\Omega)}{\textup{Vol}_g(M)}=\frac{\textup{Vol}_{g_0}(\Omega^{\ast})}{\textup{Vol}_{g_0}(\mathbb{S}^d(R))},
\end{equation}
then
\begin{equation}\label{Ineq:IsoIneq}
\frac{\textup{Surf}_g(\partial \Omega)}{\textup{Vol}_g(M)}\geq\frac{\textup{Surf}_{g_0}(\partial \Omega^{\ast})}{\textup{Vol}_{g_0}(\mathbb{S}^d(R))}.
\end{equation}
Here, $g_0$ denotes the canonical metric on $\mathbb{S}^d(R)$. Moreover, the radius $R$ depends only on $K$, the dimension $d$, and the diameter of $M$. Specifically, we have
\[
R=
\begin{cases}
\frac{1}{\sqrt{K}}\left( \frac{2\int_0^{\frac{\textup{diam}(\Omega)\sqrt{K}}{2}}\cos^{d-1}\theta\,d\theta}{\int_0^{\pi}\sin^{d-1}\theta\,d\theta}\right)^{\frac{1}{d}}& \textup{if }K>0,\\
\frac{\textup{diam}(\Omega)}{\left(1+d\int_0^{\pi}\sin^{d-1}\theta \,d\theta \right)^{\frac{1}{d}}-1}& \textup{if }K=0,\\
\frac{1}{\sqrt{-K}C(\textup{diam}(\Omega)\sqrt{-K})}& \textup{if }K<0,\\
\end{cases}
\]
where $C(z)$ denotes the unique positive solution $x$ of the equation
\[
x\int_0^{z}\left(\cosh t+x\sinh t\right)^{d-1}\,dt=\int_0^{\pi}\sin^{d-1}\theta \,d\theta.
\]
\end{theorem}
%\marginnote{Should we even mention the explicit formula for $R$? It may be unnecessary.}

If $K>0$, the isoperimetric inequality is sharp in the following sense: from Myers's Theorem \cite{M} and the explicit formula for $R$, we see $R\leq \frac{1}{\sqrt{K}}$. If $\Omega^{**}$ denotes a geodesic ball in the sphere $\mathbb S^d \left(\frac{1}{\sqrt{K}}\right)$ satisfying $\frac{\textup{Vol}_g(\Omega)}{\textup{Vol}_g(M)} = \frac{\textup{Vol}_{g_1}(\Omega^{**})}{\textup{Vol}_{g_1}\left(\mathbb{S}^d\left(\frac{1}{\sqrt{K}}\right)\right)}$, then
\begin{equation}\label{Ineq:IsoIneq2}
\frac{\textup{Surf}_g(\partial \Omega)}{\textup{Vol}_g(M)}\geq\frac{\textup{Surf}_{g_1}(\partial \Omega^{\ast \ast})}{\textup{Vol}_{g_1}\left(\mathbb{S}^d\left(\frac{1}{\sqrt{K}}\right)\right)},
\end{equation}
where $g_1$ denotes the canonical metric on $\mathbb S^d \left(\frac{1}{\sqrt{K}}\right)$. When equality holds in \eqref{Ineq:IsoIneq2}, Cheng's Theorem \cite{C} implies that $(M,g)$ is isometric to the sphere $\left(\mathbb{S}^d\left(\frac{1}{\sqrt{K}}\right),g_1\right)$ (and $\Omega$ is isometric to a geodesic ball). In what follows, we refer to this observation as \emph{the equality case} of Theorem \ref{Thm:IsoIneq}.

This isoperimetric inequality and the notion of spherical symmetrization play starring roles in the proofs of our comparison results. To define the latter notion, take $f\in L^1(\Omega)$ non-negative and define the \emph{distribution function of $f$} by
\[
\mu_f(t)=\textup{Vol}_g\left(\{x\in \Omega:f(x)>t\} \right),\quad t\in \mathbb{R}.
\]
The \emph{decreasing rearrangement} $f^{\#}:[0,\textup{Vol}_g(\Omega)]\to \mathbb{R}$ is then defined using the distribution function:
\[
f^{\#}(t)=
\begin{cases}
\underset{\Omega}{\textup{ess sup}}\,f & \textup{if }t=0,\\
\inf \{s:\mu_f(s)\leq t \} & \textup{if }t>0.
\end{cases}
\]
Finally, to define the \emph{spherical symmetrization} $f^{\ast}:\Omega^{\ast} \to \mathbb{R}$ we fix a pole $x_0\in \mathbb{S}^d(R)$ and use the decreasing rearrangement:
\begin{equation}\label{Eqn:DefDR}
f^{\ast}(x)=f^{\#}\left( \frac{\textup{Vol}_g(M)}{\textup{Vol}_{g_0}(\mathbb{S}^d(R))} \textup{Vol}_{g_0}\left(B(r)\right)\right),
\end{equation}
where $B(r)$ denotes the geodesic ball on $\mathbb{S}^d(R)$ centered at $x_0$ of radius $r=\textup{dist}_{g_0}(x,x_0)$.

The spherical symmetrization $f^{\ast}$ is a ``rearrangement'' of $f$ in the sense that
\begin{equation}\label{Eqn:RearrMean}
\frac{\mu_{f}(t)}{\textup{Vol}_g(\Omega)}=\frac{\mu_{f^{\ast}}(t)}{\textup{Vol}_{g_0}(\Omega^{\ast})}
\end{equation}
for each $t\in \mathbb{R}$. This equation essentially says that $f$ and $f^{\ast}$ have the same size. For instance, multiplying both sides by $pt^{p-1}$ and integrating from $0$ to $\underset{\Omega}{\textup{ess sup}}\,f$ yields
\begin{equation}\label{Eqn:ffstar}
\frac{1}{\textup{Vol}_g(\Omega)} \int_{\Omega}f^p\,dV_g=\frac{1}{\textup{Vol}_{g_0}(\Omega^{\ast})}\int_{\Omega^{\ast}}(f^{\ast})^p\,dV_{g_0},\quad 1\leq p<\infty.
\end{equation}

When the lower Ricci curvature bound satisfies $K>0$, we shall consider a second spherical symmetrization defined on $\mathbb{S}^d\left(\frac{1}{\sqrt{K}}\right)$. With $f\in L^1(\Omega)$ as above, we similarly fix a pole $x_1\in \mathbb{S}^d\left(\frac{1}{\sqrt{K}}\right)$ and define  
\begin{equation}\label{Eqn:DefDR2}
f^{\ast \ast}(x)=f^{\#}\left( \frac{\textup{Vol}_g(M)}{\textup{Vol}_{g_1}\left(\mathbb{S}^d\left(\frac{1}{\sqrt{K}}\right)\right)} \textup{Vol}_{g_1}\left(B(r)\right)\right),
\end{equation}
where $B(r)$ denotes the geodesic ball on $\mathbb{S}^d\left(\frac{1}{\sqrt{K}}\right)$ centered at $x_1$ of radius $r=\textup{dist}_{g_1}(x,x_1)$. Formulas analogous to \eqref{Eqn:RearrMean} and \eqref{Eqn:ffstar} also hold for the symmetrization $f^{\ast \ast}$.

Before proceeding to our main results, we pause to explain our consideration of two spheres in the case of positive Ricci curvature. Here, our comparison results for solutions to PDE (Theorem \ref{Thm:CR}), moment spectra (Theorem \ref{Thm:MSI}), and eigenvalues (Corollary \ref{Cor:FK}) compare geometric data on $M$ with geometric data on both $\mathbb{S}^d(R)$ and $\mathbb{S}^d\left(\frac{1}{\sqrt{K}}\right)$; the comparison with $\mathbb{S}^d(R)$ is always stronger. However, we include the comparison with $\mathbb{S}^d\left(\frac{1}{\sqrt{K}}\right)$ because there, we are able to handle sharp cases of equality.

\section{Main Results}
We start by collecting some basic facts about the moment spectrum on complete Riemannian manifolds.
\begin{theorem}\label{Th1}
Let $M$ be a complete Riemannian manifold and $\Omega\subseteq M$ a smoothly bounded precompact domain. Then the moment spectrum $T_n(\Omega)$ determines $\textup{spec}^{\ast}(\Omega)$ and the volume partition $\{a_{\nu}^2\}_{\nu \in \textup{spec}^{\ast}(\Omega)}.$  More precisely, for two successive elements $\nu_k,\nu_{k+1}\in \textup{spec}^{\ast}(\Omega)$, we have

\begin{equation*}
\frac{1}{\nu_k}=\lim_{n\to \infty} \left[\frac{T_n(\Omega)}{n!}-\sum_{\underset{\nu<\nu_k}{\nu \in \textup{spec}^{\ast}(\Omega)}} a_{\nu}^2 \left(\frac{1}{\nu}\right)^n\right]^{\frac{1}{n}}
\end{equation*}
and
\[
\frac{\nu_k}{\nu_{k+1}}=\lim_{n\to \infty} \left[ \nu_k^n\left(\frac{T_n(\Omega)}{n!}-\sum_{\underset{\nu<\nu_k}{\nu \in \textup{spec}^{\ast}(\Omega)}} a_{\nu}^2 \left(\frac{1}{\nu}\right)^n\right)-a_{\nu_k}\right]^{\frac{1}{n}}.
\]
In particular, since $\frac{\nu_k}{\nu_{k+1}}<1$, we have
\[
a_{\nu_k}=\lim_{n\to \infty} \nu_k^n \left(\frac{T_n(\Omega)}{n!}-\sum_{\underset{\nu<\nu_k}{\nu \in \textup{spec}^{\ast}(\Omega)}} a_{\nu}^2 \left(\frac{1}{\nu}\right)^n\right).
\]

\end{theorem}

\begin{proof}
By \eqref{Eqn:vp2} the $a_{\nu}^2$ are bounded, say $a_{\nu}^2\leq N$ for $\nu \in \textup{spec}^{\ast}(\Omega)$. We therefore have
\begin{equation}\label{Eqn:bounds}
a_{\nu_1}^{\frac{2}{n}} \frac{1}{\nu_1} \leq \left(\sum_{\nu \in \textup{spec}^{\ast}(\Omega)}a_{\nu}^2 \left(\frac{1}{\nu}\right)^n\right)^{\frac{1}{n}} \leq N^{\frac{2}{n}}\left(\sum_{\nu \in \textup{spec}^{\ast}(\Omega)}  \left(\frac{1}{\nu}\right)^n\right)^\frac{1}{n}.
\end{equation}
Since
\[
\lim_{n\to \infty}\left(\sum_{\nu \in \textup{spec}^{\ast}(\Omega)}  \left(\frac{1}{\nu}\right)^n\right)^\frac{1}{n}=\sup_{\nu \in \textup{spec}^{\ast}(\Omega)} \frac{1}{\nu}=\frac{1}{\nu_1},
\]
letting $n\to \infty$ in \eqref{Eqn:bounds} and combining with \eqref{Eqn:Zeta} and \eqref{Eqn:ZetaTn} yields 
\[
\lim_{n\to \infty}\left(\frac{T_n(\Omega)}{n!}\right)^{\frac{1}{n}}=\frac{1}{\nu_1}.
\]
We conclude that $\nu_1$ is determined by the moment spectrum, and so too is
\[
\nu_1^n\zeta(n)=\frac{\nu_1^nT_n(\Omega)}{n!}=a_{\nu_1}^2+\sum_{\underset{\nu>\nu_1}{\nu \in \textup{spec}^{\ast}(\Omega)}} a_{\nu}^2 \left(\frac{\nu_1}{\nu}\right)^n.
\]
Arguing as above, we conclude
\[
\lim_{n\to \infty}\left( \frac{\nu_1^nT_n(\Omega)}{n!}-a_{\nu_1}^2\right)^{\frac{1}{n}}=\frac{\nu_1}{\nu_2}.
\]
Since $\frac{\nu_1}{\nu_2}<1$, we deduce
\[
\lim_{n\to \infty} \frac{\nu_1^nT_n(\Omega)}{n!}=a_{\nu_1}^2,
\]
showing that $a_{\nu_1}^2$ is determined by the moment spectrum. Having established that both $\nu_1$ and $a_{\nu_1}^2$ are determined by the moment spectrum, the same holds true for
\[
\zeta(n)-a_{\nu_1}^2\left(\frac{1}{\nu_1}\right)^n= \sum_{\underset{\nu>\nu_1}{\nu \in \textup{spec}^{\ast}(\Omega)}} a_{\nu}^2 \left(\frac{1}{\nu}\right)^n.
\]
Arguing exactly as above, we deduce that
\[
\lim_{n\to \infty} \left[\frac{T_n(\Omega)}{n!}- a_{\nu_1}^2 \left(\frac{1}{\nu_1}\right)^n\right]^{\frac{1}{n}}=\frac{1}{\nu_2}.
\]
We likewise deduce
\[
\lim_{n\to \infty} \left[\nu_2^n\left(\frac{T_n(\Omega)}{n!}-a_{\nu_1}^2\left( \frac{1}{\nu_1}\right)^n\right)-a_{\nu_2}^2\right]^{\frac{1}{n}}=\frac{\nu_2}{\nu_3}
\]
and
\[
\lim_{n\to \infty} \nu_2^n\left(\frac{T_n(\Omega)}{n!}-a_{\nu_1}^2 \left(\frac{1}{\nu_1}\right)^n\right)=a_{\nu_2}^2.
\]
We conclude that $\nu_2$ and $a_{\nu_2}^2$ are determined by the moment spectrum. The general claims and formulas follow by iterating this argument.
\end{proof}

An immediate consequence of Theorem \ref{Th1} is (see also \cite{MM}):

\begin{corollary}
The moment spectrum determines heat content.
\end{corollary}

\begin{proof}
This follows immediately from \eqref{Eqn:HC}.
\end{proof}

We next establish the estimate of Theorem \ref{Thm:LnEst}:

\begin{proof}[Proof of Theorem \ref{Thm:LnEst}.]
Replacing $\lambda_n$ by the lowest equivalent eigenvalue, we may assume $\lambda_{n-1}<\lambda_n$. Let $v_1,v_2,\ldots, v_{n-1}$ denote a corresponding set of orthonormal eigenfunctions for the eigenvalues $\lambda_1,\lambda_2,\ldots,\lambda_{n-1}$. Define
\[
u=u_k-\sum_{j=1}^{n-1}(u_k,v_j)v_j,
\]
where $u_k$ solves \eqref{Eqn:PoisHier} and $(\cdot,\cdot)$ denotes the standard inner product on $L^2(\Omega)$. Using $u$ as a trial function in the Rayleigh quotient for $\lambda_n$, we deduce
\begin{equation}\label{Ineq:LNRQ}
\lambda_n \leq \frac{\displaystyle \int_{\Omega}|\nabla u|^2\,dV_g}{\displaystyle \int_{\Omega}u^2\,dV_g}=\frac{\displaystyle \int_{\Omega}|\nabla u_k|^2\,dV_g-\displaystyle \sum_{j=1}^{n-1}\lambda_j(u_k,v_j)^2}{\displaystyle \int_{\Omega}u_k^2\,dV_g-\displaystyle \sum_{j=1}^{n-1}(u_k,v_j)^2}.
\end{equation}
To simplify the numerator, we integrate by parts:
\begin{align*}
\int_{\Omega}|\nabla u_k|^2\,dV_g&=-\int_{\Omega}u_k\Delta u_k\,dV_g\\
&=k\int_{\Omega}u_ku_{k-1}\,dV_g\\
&=-\frac{k}{k+1}\int_{\Omega}\Delta u_{k+1}u_{k-1}\,dV_g\\
&=-\frac{k}{k+1}\int_{\Omega} u_{k+1}\Delta u_{k-1}\,dV_g\\
&=\frac{k(k-1)}{k+1}\int_{\Omega}u_{k+1}u_{k-2}\,dV_g.
\end{align*}
Iterating this process, we see
\begin{equation}\label{Eqn:NFT}
\int_{\Omega}|\nabla u_k|^2\,dV_g=\frac{(k!)^2}{(2k-1)!}T_{2k-1}(\Omega).
\end{equation}
To further simplify \eqref{Ineq:LNRQ}, we compute
\begin{align*}
\int_{\Omega}u_kv_j\,dV_g&=-\frac{1}{\lambda_j}\int_{\Omega}u_k\Delta v_j\,dV_g\\
&=-\frac{1}{\lambda_j}\int_{\Omega}\Delta u_k v_j\,dV_g\\
&=\frac{k}{\lambda_j}\int_{\Omega}u_{k-1}v_j\,dV_g.\\
\end{align*}
Iterating this argument gives
\begin{equation}\label{Eqn:SNE}
\int_{\Omega}u_kv_j\,dV_g=\frac{k!}{\lambda_j^k}\int_{\Omega}v_j\,dV_g.
\end{equation}
Fix an eigenvalue $\nu$ from $\lambda_1,\lambda_2,\ldots ,\lambda_{n-1}$ and let $\textup{proj}\,_{E_{\nu}}1$ denote the orthogonal projection of the constant function $1$ onto the eigenspace $E_{\nu}$. We then have
\[
\textup{proj}\,_{E_{\nu}}1=\underset{\lambda_j=\nu}{\sum_{j=1}^{n-1}}(1,v_j)v_j.
\]
It follows from \eqref{Eqn:SNE} that
\begin{align}
\underset{\lambda_j=\nu}{\sum_{j=1}^{n-1}}(u_k,v_j)^2&=\left(\frac{k!}{\nu^k}\right)^2 \underset{\lambda_j=\nu}{\sum_{j=1}^{n-1}}(1,v_j)^2 \nonumber\\
&=\left(\frac{k!}{\nu^k}\right)^2\int_{\Omega}\left(\textup{proj}\,_{E_{\nu}}1 \right)^2\,dV_g.\label{Eqn:alambd}
\end{align}
We finally simplify the remaining term in the denominator of \eqref{Ineq:LNRQ}:
\begin{align*}
\int_{\Omega}u_k^2\,dV_g&=-\frac{1}{k+1}\int_{\Omega}\Delta u_{k+1}u_k\,dV_g\\
&=-\frac{1}{k+1}\int_{\Omega} u_{k+1}\Delta u_k\,dV_g\\
&=\frac{k}{k+1}\int_{\Omega}u_{k+1}u_{k-1}\,dV_g,
\end{align*}
and repeated application of this argument yields
\begin{equation}\label{Eqn:TEN}
\int_{\Omega}u_k^2\,dV_g=\frac{(k!)^2}{(2k)!}T_{2k}(\Omega).
\end{equation}
The claimed inequality \eqref{Ineq:LambdanEst} follows by using \eqref{Eqn:NFT}, \eqref{Eqn:alambd}, and \eqref{Eqn:TEN} in \eqref{Ineq:LNRQ} and using the definition of $a_{\nu}^2$.

To establish the equality claim, assume $\lambda_n=\nu_m\in \textup{spec}^{\ast}(\Omega)$. We use \eqref{Eqn:Zeta} and \eqref{Eqn:ZetaTn}, keeping only the first term of the denominator to estimate
\begin{align}
\frac{ \displaystyle \frac{T_{2k-1}(\Omega)}{(2k-1)!} - \sum_{\underset{\nu <\lambda_n}{\nu \in \textup{spec}^{\ast}(\Omega)}} a_{\nu}^2\left(\frac{1}{\nu}\right)^{2k-1}}{\displaystyle \frac{T_{2k}(\Omega)}{(2k)!} -  \sum_{\underset{\nu <\lambda_n}{\nu \in \textup{spec}^{\ast}(\Omega)}} a_{\nu}^2\left(\frac{1}{\nu}\right)^{2k}}
&= \frac{\displaystyle \sum_{\underset{\nu \geq \lambda_n}{\nu \in \textup{spec}^{\ast}(\Omega)}} a_{\nu}^2\left(\frac{1}{\nu}\right)^{2k-1}}{\displaystyle \sum_{\underset{\nu \geq \lambda_n}{\nu \in \textup{spec}^{\ast}(\Omega)}} a_{\nu}^2\left(\frac{1}{\nu}\right)^{2k}} \nonumber \\
& \leq \lambda_n\frac{\displaystyle \sum_{\underset{\nu \geq \lambda_n}{\nu \in \textup{spec}^{\ast}(\Omega)}} a_{\nu}^2\left(\frac{1}{\nu}\right)^{2k-1}}{a_{\lambda_n}^2\left(\frac{1}{\lambda_n}\right)^{2k-1}}\label{Ineq:FstLnEst}.
\end{align}
We further estimate
\begin{align}
\lambda_n\frac{\displaystyle \sum_{\underset{\nu \geq \lambda_n}{\nu \in \textup{spec}^{\ast}(\Omega)}} a_{\nu}^2\left(\frac{1}{\nu}\right)^{2k-1}}{a_{\lambda_n}^2\left(\frac{1}{\lambda_n}\right)^{2k-1}} &=  \lambda_n \left(1+\frac{1}{a_{\lambda_n}^2}\sum_{\underset{\nu > \lambda_n}{\nu \in \textup{spec}^{\ast}(\Omega)}} a_{\nu}^2\left(\frac{\lambda_n}{\nu}\right)^{2k-1} \right) \nonumber \\
&\leq \lambda_n \left(1+\frac{1}{a_{\lambda_n}^2}\left(\frac{\lambda_n}{\nu_{m+1}}\right)^{2k-1}\sum_{\underset{\nu > \lambda_n}{\nu \in \textup{spec}^{\ast}(\Omega)}} a_{\nu}^2 \right)\nonumber \\
&\leq \lambda_n \left(1+\frac{\textup{Vol}_g(\Omega)}{a_{\lambda_n}^2}\left(\frac{\lambda_n}{\nu_{m+1}}\right)^{2k-1}\right), \label{Ineq:FinLnEst}
\end{align}
where the last inequality follows from \eqref{Eqn:vp2}. Letting $k\to \infty$ and combining \eqref{Ineq:LambdanEst}, \eqref{Ineq:FstLnEst}, and \eqref{Ineq:FinLnEst} gives the result.

\end{proof}

The techniques used to prove Theorem \ref{Thm:LnEst} can also be used to establish estimates for the moment spectrum in terms of the manifold's Cheeger constant. Recall that for compact manifolds $M$, the  \emph{Cheeger constant} $C$ is defined by
\[
C=\inf_{\Omega}\frac{\textup{Surf}_g(\partial \Omega)}{\min \{\textup{Vol}_g(\Omega),\textup{Vol}_g(M\setminus \Omega)\}},
\]
where the $\inf$ ranges over all smoothly bounded domains in $M$.

For the next result (and the remainder of the paper), we shall make use of the following shorthand notation for functions $u:\Omega \to \mathbb{R}$:
\begin{align*}
\{u>t\}&=\{x\in \Omega:u(x)>t\},\\
\{u=t\}&=\{x\in \Omega:u(x)=t\}.
\end{align*}
We have the following estimate:
\begin{theorem}\label{Thm:CSE}
Let $M$ be a connected compact Riemannian manifold with $\Omega\subseteq M$ a smoothly bounded domain. If $\textup{Vol}_g(\Omega)\leq \frac{1}{2}\textup{Vol}_g(M)$, then
\[
C^2\leq \textup{Vol}_g(\Omega)
\frac{(k!)^2}{(2k-1)!}\frac{T_{2k-1}(\Omega)}{T_k(\Omega)^2}.
\]
\end{theorem}
%\marginnote{Should we provide references for the coarea formula or Sard's Theorem?}
\begin{proof}
Our argument follows \cite{CGL}. By definition, with $u_k$ as in \eqref{Eqn:PoisHier},
\begin{equation}\label{Eqn:TkR}
T_k(\Omega)=\int_{\Omega}u_k\,dV_g =\int_0^{\infty}\mu_{u_k}(t)\,dt.
\end{equation}
From our assumption on the volume of $\Omega$, we see $\mu_{u_k}(t)\leq \frac{1}{2}\textup{Vol}_g(M)$. Moreover, because $u_k$ is smooth, it follows from Sard's Theorem that
\[
\partial \{u_k>t\}=\{u_k=t\}
\]
for almost every $t\geq 0$. For such $t$, it follows that
\[
\textup{Surf}_g(\{u_k=t\})\geq C\mu_{u_k}(t),
\]
where $C$ is the Cheeger constant. Invoking the coarea formula and Cauchy-Schwarz, \eqref{Eqn:TkR} becomes
\begin{align}
T_k(\Omega)&\leq \frac{1}{C}\int_0^{\infty} \textup{Surf}_g(\{u_k=t\})\,dt \nonumber \\
&=\frac{1}{C}\int_{\Omega}|\nabla u_k|\,dV_g \nonumber \\
&\leq \frac{\textup{Vol}_g(\Omega)^{\frac{1}{2}}}{C}\left(\int_{\Omega}|\nabla u_k|^2\,dV_g\right)^\frac{1}{2}.\label{Ineq:TkEst}
\end{align}
The argument used in the proof of Theorem  \ref{Thm:LnEst} gives
\[
\int_{\Omega}|\nabla u_k|^2\,dV_g=\frac{(k!)^2}{(2k-1)!}T_{2k-1}(\Omega).
\]
Substituting this equality into \eqref{Ineq:TkEst} gives the result.
\end{proof}

We now turn our attention to estimates that involve lower Ricci curvature bounds. Before proceeding, the reader may find it useful to review the definitions and notation introduced in Section \ref{SubSect:Symm}. Our first result is the following PDE comparison principle:

\begin{theorem}\label{Thm:CR}
Let $M$, $\Omega$, $\mathbb{S}^d(R)$, $\Omega^{\ast}$, $\mathbb{S}^d\left(\frac{1}{\sqrt{K}}\right)$, and $\Omega^{**}$ be as in Theorem \ref{Thm:MSI}. Let $f\geq 0$ be a continuous function on $\Omega$ and assume $u$ and $v$ are smooth solutions of the Poisson problems
\[
-\Delta_g u=f\quad \textup{in }\Omega,\qquad u=0\quad \textup{on }\partial \Omega,
\]
and
\[
-\Delta_{g_0} v=f^{\ast} \quad \textup{in }\Omega^{\ast},\qquad v=0\quad \textup{on }\partial \Omega^{\ast},
\]
where $f^{\ast}$ denotes the spherical symmetrization of $f$ as defined by \eqref{Eqn:DefDR}. Then $u^{\ast}\leq v$ in $\Omega^{\ast}$.

Moreover, if $K>0$ and $w$ is a smooth solution to the Poisson problem
\[
-\Delta_{g_1} w=f^{\ast \ast} \quad \textup{in }\Omega^{\ast \ast},\qquad w=0\quad \textup{on }\partial \Omega^{\ast \ast},
\]
with $f^{\ast \ast}$ the spherical symmetrization of $f$ defined by \eqref{Eqn:DefDR2}, then $u^{\ast \ast}\leq w$. If $u^{\ast \ast}= w$, then $M$ is isometric to the sphere $\mathbb{S}^d\left(\frac{1}{\sqrt{K}}\right)$ and $\Omega$ is isometric to an appropriate geodesic ball in $\mathbb{S}^d\left(\frac{1}{\sqrt{K}}\right)$.

\end{theorem}

\begin{proof}
We first claim $v=v^{\ast}$. Denote
\[
s_R(t)=R\sin\left(\frac{t}{R}\right)
\]
and observe that $v$ solves the ODE
\[
-\Delta_{g_0}v=-s_R^{1-d}(r)\frac{\partial}{\partial r}\left(s_R^{d-1}(r)\frac{\partial v}{\partial r}\right)=f^{\ast}(r),\qquad \frac{\partial v}{\partial r}(0)=v(R_0)=0,
\]
where $R_0$ denotes the radius of the geodesic ball $\Omega^*$.
Writing
\[
F(w)=\int_0^wf^{\#}(z)\,dz,
\]
it follows that
\begin{align}
v(r)&=\int_{r}^{R_0}s_R^{1-d}(\tau)\int_0^{\tau}s_R^{d-1}(\xi)f^{\ast}(\xi)\,d\xi \,d\tau \nonumber \\
&=\frac{\textup{Vol}_{g_0}(\mathbb{S}^d(R))}{\beta_{d-1}\textup{Vol}_g(M)}\int_r^{R_0}s_R^{1-d}(\tau)F\left(\frac{\textup{Vol}_g(M)}{\textup{Vol}_{g_0}(\mathbb{S}^d(R))}\textup{Vol}_{g_0}(B(\tau))\right)\,d\tau, \label{Eqn:vrep}
\end{align}
where $\beta_{d-1}$ denotes the surface measure of the unit $(d-1)$-sphere and $B(\tau)$ denotes a geodesic ball in $\mathbb S^d(R)$ of radius $\tau$. From this representation for $v$, it follows that $\frac{\partial v}{\partial r}\leq 0$ since $0\leq f$, and so $v=v^{\ast}$.

Define a function $r(t)$ using the equality of sets
\begin{equation}\label{Eqn:Defrfctn}
\{u>t\}^{\ast}=\{v>r(t)\},
\end{equation}
%\begin{equation}\label{Eqn:VC}
%\frac{\textup{Vol}_g(\Omega_t)}{\textup{Vol}_g(M)}=\frac{\textup{Vol}_{g_0}(\Omega_{r(t)}^{\#})}{\textup{Vol}_{g_0}(M^{\ast})}.
%\end{equation}
and observe that $r(t)$ is strictly increasing on $(0,\underset{\Omega}{\textup{ess sup}}\,u)$. As in the proof of Theorem \ref{Thm:CSE}, it follows from Sard's Theorem that
\[
\partial \{u>t\}=\{u=t\}
\]
for almost every $t\geq 0$ (and similarly for $v$). By the same result, the following integrals involving the gradient are well-defined for almost every $t\geq 0$. From Cauchy-Schwarz,
\begin{equation}\label{Ineq:CS}
 \int_{\{u=t\}}\frac{1}{|\nabla u|}\,dS_g\geq \frac{\left(\textup{Surf}_g(\{u=t\})\right)^2}{\int_{\{u=t\}} |\nabla u| \,dS_g}.
\end{equation}
The Divergence Theorem and \eqref{Eqn:ffstar} give
\begin{align*}
\int_{\{u=t\}}|\nabla u|\,dS_g &=\int_{\{u>t\}}f\,dV_g\\
& = \frac{\textup{Vol}_{g}(\{u>t\})}{\textup{Vol}_{g_0}(\{v>r(t)\})}\int_{\{v>r(t)\}}(f \big|_{\{u>t\}})^{\ast}\,dV_{g_0} \\
& \leq \frac{\textup{Vol}_{g}(\{u>t\})}{\textup{Vol}_{g_0}(\{v>r(t)\})}\int_{\{v>r(t)\}}f^{\ast}\,dV_{g_0} \\
&=\frac{\textup{Vol}_{g}(\{u>t\})}{\textup{Vol}_{g_0}(\{v>r(t)\})} \int_{\{v=r(t)\}}|\nabla v|\,dS_{g_0},
\end{align*}
and combining with \eqref{Ineq:CS}, we have
\begin{align}
\int_{\{u=t\}}\frac{1}{|\nabla u|}\,dS_{g} & \geq \frac{\left(\textup{Surf}_g(\{u=t\})\right)^2}{\textup{Vol}_g(\{u>t\})} \frac{\textup{Vol}_{g_0}(\{v>r(t)\})}{\left(\textup{Surf}_{g_0}(\{v=r(t)\})\right)^2}  \int_{\{v=r(t)\}} \frac{1}{|\nabla v|}\,dS_{g_0} \nonumber \\
&\geq \frac{\textup{Vol}_{g}(M)}{\textup{Vol}_{g_0}(\mathbb{S}^d(R))}\int_{\{v=r(t)\}}\frac{1}{|\nabla v|}\,dS_{g_0}.\label{Ineq:r}
\end{align}
The first inequality follows since $|\nabla v|$ is constant on $\{v=r(t)\}$ and the second inequality follows from Theorem \ref{Thm:IsoIneq}. On the other hand, combining \eqref{Eqn:Defrfctn} with the coarea formula, we see
\[
\int_t^{\infty}\int_{\{u=s\}}\frac{1}{|\nabla u|}\,dS_g\,ds=\frac{\textup{Vol}_{g}(M)}{\textup{Vol}_{g_0}(\mathbb{S}^d(R))}\int_{r(t)}^{\infty}\int_{\{v=s\}}\frac{1}{|\nabla v|}\,dS_{g_0}\,ds.
\]
Differentiating both sides with respect to $t$ gives
\[
\int_{\{u=t\}}\frac{1}{|\nabla u|}\,dS_{g}=r'(t)\frac{\textup{Vol}_{g}(M)}{\textup{Vol}_{g_0}(\mathbb{S}^d(R))}\int_{\{v=r(t)\}}\frac{1}{|\nabla v|}\,dS_{g_0},
\]
and combining with \eqref{Ineq:r} we find $r'(t)\geq 1$. Since $r(0)=0$, we have $r(t)\geq t$ which implies
\[
\mu_u(t)\leq \frac{\textup{Vol}_{g}(M)}{\textup{Vol}_{g_0}(\mathbb{S}^d(R))}\mu_v(t).
\]
Using the definition of spherical symmetrization, we immediately deduce $u^{\ast}\leq v^{\ast}$, and having already established that $v^{\ast}=v$, the claimed inequality follows.

For the remainder of the proof, we assume $K>0$. Since $\mathbb{S}^d\left(\frac{1}{\sqrt{K}}\right)$ obeys an isoperimetric inequality \eqref{Ineq:IsoIneq2}, all of our work above still holds if we replace $R$ by $\frac{1}{\sqrt{K}}$, $v$ by $w$, $g_0$ by $g_1$, and $\ast$ by $\ast \ast$. In particular, $u^{\ast \ast}\leq w$. For the case of equality, we adapt the techniques of Kesavan \cite{K}. If $u^{\ast \ast}=w$, it follows that
\begin{equation}\label{Eqn:EqnDFs}
\mu_u(t)= \frac{\textup{Vol}_{g}(M)}{\textup{Vol}_{g_1}\left(\mathbb{S}^d\left(\frac{1}{\sqrt{K}}\right)\right)}\mu_w(t)
\end{equation}
for each $t\geq 0$. Using the coarea formula and Cauchy-Schwarz, for almost every $t\geq 0$ we have
\begin{align}
\left(\textup{Surf}_{g}(\{u=t\})\right)^2&=\left(\frac{d}{dt}\int_{\{u>t\}}|\nabla u|\,dV_g \right)^2 \nonumber \\
&\leq -\mu_u'(t)\int_{\{u>t\}}f\,dV_g\nonumber \\
&\leq -\mu_u'(t)F\left(\mu_u(t)\right) \label{Ineq:PerEst}.
\end{align}
Returning our attention to \eqref{Eqn:vrep}, for each $t\in[0,\underset{{\Omega^{\ast \ast}}}{\textup{ess sup}}  \ w]$ let $\rho(t)$ satisfy
\begin{equation}\label{Eqn:Defrho}
t=\frac{\textup{Vol}_{g_1}\left(\mathbb{S}^d\left(\frac{1}{\sqrt{K}}\right)\right)}{\beta_{d-1}\textup{Vol}_g(M)}\int_{\rho(t)}^{R_1}s_{\frac{1}{\sqrt{K}}}^{1-d}(\tau)F\left(\frac{\textup{Vol}_g(M)}{\textup{Vol}_{g_1}\left(\mathbb{S}^d\left(\frac{1}{\sqrt{K}}\right)\right)}\textup{Vol}_{g_1}(B(\tau))\right)\,d\tau,
\end{equation}
where $R_1$ denotes the radius of the geodesic ball $\Omega^{\ast \ast}$ and $B(\tau)$ denotes a geodesic ball in $\mathbb S^d \left(\frac{1}{\sqrt{K}}\right)$ of radius $\tau$. Differentiating  both sides with respect to $t$ gives
\begin{equation}\label{Eqn:Diff1Eqls}
1=-\frac{\textup{Vol}_{g_1}\left(\mathbb{S}^d\left(\frac{1}{\sqrt{K}}\right)\right)}{\beta_{d-1}\textup{Vol}_g(M)}s_{\frac{1}{\sqrt{K}}}^{1-d}(\rho(t))F\left(\frac{\textup{Vol}_g(M)}{\textup{Vol}_{g_1}\left(\mathbb{S}^d\left(\frac{1}{\sqrt{K}}\right)\right)}\textup{Vol}_{g_1}(B(\rho(t))) \right)\rho'(t).
\end{equation}
Unless $f=0$, $w$ is strictly decreasing in $r$, and so it follows from \eqref{Eqn:Defrho} that
\begin{equation}\label{Eqn:VLS}
\{w>t\}=B(\rho(t)).
\end{equation}
Using polar coordinates, we see
\[
\mu_w(t)=\beta_{d-1}\left(\frac{1}{\sqrt{K}}\right)^{d}\int_0^{\sqrt{K}\rho(t)}\sin^{d-1}\theta \,d\theta,
\]
and differentiation yields
\begin{align}
\mu_w'(t)&=\beta_{d-1}\left(\frac{1}{\sqrt{K}}\right)^{d-1}\sin^{d-1}\left(\sqrt{K}\rho(t)\right)\rho'(t) \nonumber \\
&=\beta_{d-1}s_{\frac{1}{\sqrt{K}}}^{d-1}(\rho(t))\rho'(t). \label{Eqn:Muv}
\end{align}
Combining \eqref{Eqn:EqnDFs}, \eqref{Eqn:Diff1Eqls}, \eqref{Eqn:VLS}, and \eqref{Eqn:Muv}, we see
\[
F\left(\mu_u(t)\right)\mu_u'(t)=-\left( \frac{\beta_{d-1}\textup{Vol}_g(M)}{\textup{Vol}_{g_1}\left(\mathbb{S}^d\left(\frac{1}{\sqrt{K}}\right)\right)}s_{\frac{1}{\sqrt{K}}}^{d-1}(\rho(t))\right)^2,
\]
and combining with \eqref{Ineq:PerEst}, we obtain
\begin{equation}\label{Eqn:EqISP}
\frac{\textup{Surf}_{g}(\{u=t\})}{\textup{Vol}_g(M)} \leq \frac{\beta_{d-1}s_{\frac{1}{\sqrt{K}}}^{d-1}(\rho(t))}{\textup{Vol}_{g_1}\left(\mathbb{S}^d\left(\frac{1}{\sqrt{K}}\right)\right)} .
\end{equation}
On the other hand, since $u^{\ast \ast}=w$, it follows from \eqref{Eqn:VLS} that $\{u>t\}^{\ast \ast}$ is a geodesic ball in $\mathbb{S}^d\left(\frac{1}{\sqrt{K}}\right)$ of radius $\rho(t)$. Hence $\partial \{u>t\}^{\ast \ast}$ is a $(d-1)$-dimensional sphere of radius $\frac{1}{\sqrt{K}}\sin \left(\sqrt{K}\rho(t)\right)$. From Theorem $\ref{Thm:IsoIneq}$, we therefore have
\[
\frac{\beta_{d-1}s_{\frac{1}{\sqrt{K}}}^{d-1}(\rho(t))}{\textup{Vol}_{g_1}\left(\mathbb{S}^d\left(\frac{1}{\sqrt{K}}\right)\right)} \leq \frac{\textup{Surf}_{g}(\{u=t\})}{\textup{Vol}_g(M)}.
\]
Combining this inequality with \eqref{Eqn:EqISP}, we deduce that for almost every $t\geq 0$, the set $\{u>t\}$ achieves equality in Theorem \ref{Thm:IsoIneq} and is therefore isometric to a geodesic ball in $\mathbb{S}^d\left(\frac{1}{\sqrt{K}}\right)$. Let $t_n$ denote a strictly decreasing sequence of such $t$-values with $t_n\to 0$ and suppose $\Phi:M \to \mathbb{S}^d\left(\frac{1}{\sqrt{K}}\right)$ is an isometry as guaranteed by Theorem \ref{Thm:IsoIneq}. Observe that
\begin{equation*}\displaystyle
\Phi(\Omega)=\Phi \left(\{u>0\}\right)=\bigcup_{n=1}^{\infty}\Phi \left(\{u>t_n\}\right)
\end{equation*}
expresses $\Phi(\Omega)$ as a nested union of geodesic balls in $\mathbb{S}^d\left(\frac{1}{\sqrt{K}}\right)$, and so we conclude that $\Phi(\Omega)$ is a geodesic ball.

 \end{proof}

We next use Theorem \ref{Thm:CR} to establish a comparison result for moment spectra.
 
 \begin{proof}[Proof of Theorem \ref{Thm:MSI}.]
 For the case $n=1$, let $u_1$ be the solution to the Poisson PDE
 \[
-\Delta_g u_1=1\quad \textup{in }\Omega,\qquad u_1=0\quad \textup{on }\partial \Omega,
\]
and let $v_1$ solve the symmetrized PDE
\[
-\Delta_{g_0} v_1=1\quad \textup{in }\Omega^{\ast},\qquad v_1=0\quad \textup{on }\partial \Omega^{\ast}.
\]
Then $u_1^{\ast}\leq v_1$ by Theorem \ref{Thm:CR}, and combining with \eqref{Eqn:ffstar}, it follows that
\[
\frac{T_1(\Omega)}{\textup{Vol}_g(\Omega)}= \frac{1}{\textup{Vol}_{g_0}(\Omega^{\ast})}\int_{\Omega^{\ast}}u_1^{\ast}\,dV_g\leq \frac{1}{\textup{Vol}_{g_0}(\Omega^{\ast})}\int_{\Omega^{\ast}}v_1\,dV_{g_0}=\frac{T_1(\Omega^{\ast})}{\textup{Vol}_{g_0}(\Omega^{\ast})}.
\]
For the case $n=2$, let $u_2$ solve
 \[
-\Delta_g u_2=2u_1\quad \textup{in }\Omega,\qquad u_2=0\quad \textup{on }\partial \Omega,
\]
let $w_2$ solve
 \[
-\Delta_{g_0} w_2=2u_1^{\ast} \quad \textup{in }\Omega^{\ast},\qquad w_2=0\quad \textup{on }\partial \Omega^{\ast},
\]
and let $v_2$ solve
\[
-\Delta_{g_0} v_2=2v_1\quad \textup{in }\Omega^{\ast},\qquad v_2=0\quad \textup{on }\partial \Omega^{\ast}.
\]
Since $u_1^{\ast} \leq v_1$, the Maximum Principle gives $w_2\leq v_2$. Applying Theorem \ref{Thm:CR} and \eqref{Eqn:ffstar} once again, we see
\begin{align*}
\frac{T_2(\Omega)}{\textup{Vol}_g(\Omega)}&=\frac{1}{\textup{Vol}_{g_0}(\Omega^{\ast})}\int_{\Omega^{\ast}}u_2^{\ast}\,dV_{g_0}\\
 & \leq \frac{1}{\textup{Vol}_{g_0}(\Omega^{\ast})}\int_{\Omega^{\ast}}w_2\,dV_{g_0}\\
  & \leq \frac{1}{\textup{Vol}_{g_0}(\Omega^{\ast})}\int_{\Omega^{\ast}}v_2\,dV_{g_0}\\
  &=\frac{T_2(\Omega^{\ast})}{\textup{Vol}_{g_0}(\Omega^{\ast})}.
\end{align*}
Inequality \eqref{Ineq:MomSp} follows by iterating the above argument repeatedly.

To establish \eqref{Ineq:MomSp2} observe that our work above still holds if we replace $R$ by $\frac{1}{\sqrt{K}}$, $g_0$ by $g_1$, and $\ast$ by $\ast \ast$. 
If for some index $n$ we have equality in \eqref{Ineq:MomSp2}, then $u_n^{\ast \ast}=v_n$, where $u_n$ and $v_n$ come from the appropriate Poisson hierarchy \eqref{Eqn:PoisHier}. The result now follows from the equality case of Theorem \ref{Thm:CR}.

\end{proof}

We finally establish the following Faber-Krahn inequality (see also \cite{BM}):

\begin{corollary}\label{Cor:FK}
Let $M$, $\Omega$, $\mathbb{S}^d(R)$, $\Omega^{\ast}$, $\mathbb{S}^d\left(\frac{1}{\sqrt{K}}\right)$, and $\Omega^{**}$ be as in Theorem \ref{Thm:MSI}. Then the lowest Dirichlet eigenvalues satisfy
\[
\lambda_1(\Omega^{\ast})\leq \lambda_1(\Omega).
\]
Moreover, if $K>0$, we also have
\[
\lambda_1(\Omega^{\ast \ast})\leq \lambda_1(\Omega), 
\]
and if equality holds, $M$ is isometric to $\mathbb{S}^d\left(\frac{1}{\sqrt{K}}\right)$ and $\Omega$ is isometric to an appropriate geodesic ball in $\mathbb{S}^d\left(\frac{1}{\sqrt{K}}\right)$.
\end{corollary}

\begin{proof}
Since any nontrivial eigenfunction for $\lambda_1(\Omega)$ has a sign, it follows that $\nu_1(\Omega)=\lambda_1(\Omega)$. Combining Theorems \ref{Thm:MSI} and \ref{Th1}, we see
\[
\frac{1}{\lambda_1(\Omega)}=\lim_{n\to \infty} \left( \frac{T_n(\Omega)}{n!}\right)^{\frac{1}{n}}\leq \lim_{n\to \infty} \left( \frac{T_n(\Omega^{\ast})}{n!}\frac{\textup{Vol}_g(\Omega)}{\textup{Vol}_{g_0}(\Omega^{\ast})}\right)^{\frac{1}{n}}=\frac{1}{\lambda_1(\Omega^{\ast})}.
\]
The proof that $\lambda_1(\Omega^{\ast \ast})\leq \lambda_1(\Omega)$ uses the same argument, invoking \eqref{Ineq:MomSp2} instead of \eqref{Ineq:MomSp}.

To establish the equality claim, we again adapt the argument of Kesavan \cite{K}. Assume $\lambda_1(\Omega)=\lambda_1(\Omega^{\ast \ast})$, and let $u$ be a $L^2$-normalized solution to the eigenvalue problem
\[
-\Delta_g u=\lambda_1(\Omega)u \quad \textup{in }\Omega,\qquad u=0\quad \textup{on }\partial \Omega.
\]
Let $w$ solve the symmetrized problem
\begin{equation}\label{Eqn:vPDE}
-\Delta_{g_1} w=\lambda_1(\Omega)u^{\ast \ast} \quad \textup{in }\Omega^{\ast \ast},\qquad w=0\quad \textup{on }\partial \Omega^{\ast \ast}.
\end{equation}
Using $w$ as a trial function in the variational characterization for $\lambda_1(\Omega^{\ast \ast})$, we see
\[
\lambda_1(\Omega^{\ast \ast})\leq \frac{\int_{\Omega^{\ast \ast}}|\nabla w|^2\,dV_{g_1}}{\int_{\Omega^{\ast \ast}}w^2\,dV_{g_1}}.
\]
On the other hand, from Theorem \ref{Thm:CR} we see $u^{\ast \ast }\leq w$, and so integration by parts gives
\begin{align*}
\int_{\Omega^{\ast \ast}}|\nabla w|^2\,dV_{g_0}&=\lambda_1(\Omega)\int_{\Omega^{\ast \ast}}wu^{\ast \ast}\,dV_{g_1}\\
& \leq \lambda_1(\Omega)\int_{\Omega^{\ast \ast}}w^2\,dV_{g_1}.
\end{align*}
Since $\lambda_1(\Omega)=\lambda_1(\Omega^{\ast \ast})$, we deduce that $w$ minimizes the Rayleigh quotient for $\lambda_1(\Omega^{\ast \ast})$ and is therefore a corresponding eigenfunction. It then follows from \eqref{Eqn:vPDE} that $u^{\ast \ast}=w$. The equality claim now follows from Theorem \ref{Thm:CR}.

\end{proof}

\end{document}